\documentclass[12pt]{article}
\usepackage{amsmath,amsthm,amsfonts,amssymb}
\usepackage{mathtools}
\usepackage{fullpage}
\usepackage{multirow}
\usepackage{array}
\usepackage{bbm}
\usepackage[noblocks]{authblk}
\usepackage{verbatim}
\usepackage{pgf,tikz,pgfplots}
\usepackage{mathrsfs}
\usetikzlibrary{arrows}
\usepackage{float}
\usepackage{enumerate}
\usepackage{diagbox}
\usepackage{color, colortbl}
\usepackage{hhline}
\usepackage{hyperref}
\usepackage[noadjust]{cite}

\newtheorem{theorem}{Theorem}[section]
\newtheorem{lemma}[theorem]{Lemma}
\newtheorem{proposition}[theorem]{Proposition}

\theoremstyle{definition}

\newlength{\Oldarrayrulewidth}

\newcommand{\N}{\mathbb{N}}

\DeclareMathOperator{\lcm}{lcm}

\begin{document}

\title{On Zeckendorf-Niven Numbers and Arithmetic Progressions}
\author[1]{Kelly~Lao\thanks{kellylao23@gmail.com}}
\author[2]{Steven~J.~Miller\thanks{sjm1@williams.edu}}
\author[3]{Nicholas~Rosa\thanks{nicholasrosa123@gmail.com}}
\author[4]{Mark~Shiliaev\thanks{mshiliaev@tamu.edu}}
\author[4]{Garrett~Tresch\thanks{treschgd@tamu.edu}}
\author[5]{Tony~W.~H.~Wong\thanks{wong@kutztown.edu}}
\author[6]{Han~Zhang\thanks{hzhang23@ole.augie.edu}}
\affil[1]{Department of Mathematics, Dartmouth College}
\affil[2]{Department of Mathematics, Williams College}
\affil[3]{Department of Mathematics, California State University, East Bay}
\affil[4]{Department of Mathematics, Texas A\&M University}
\affil[5]{Department of Mathematics, Kutztown University of Pennsylvania}
\affil[6]{Department of Mathematics, Augustana University}
\date{\today}

\maketitle

\begin{abstract}
A positive integer is Zeckendorf-Niven (respectively, Lucas-Niven) if it is divisible by the number of summands in its Zeckendorf decomposition (respectively, Lucas decomposition). We show that there exist infinitely many Zeckendorf-Niven numbers and Lucas-Niven numbers in every arithmetic progression. Furthermore, we provide bounds on the maximum number of consecutive Zeckendorf-Niven terms in certain arithmetic progressions.
\end{abstract}

\section{Introduction}

A \emph{Niven number} is a positive integer that is divisible by the sum of its digits in base $10$ representation. This definition was first introduced by Niven during a lecture in 1977. Since then, Niven numbers have been studied in the literature in various context \cite{ck1988,ck1993,costello,ddk,kennedy,kgb,mcdaniel}. Moreover, the definition of Niven numbers has also been generalized and modified in many different ways. For example, Grundman \cite{g1994} defined $b$-Niven numbers as positive integers that are divisible by their digit sums in base $b$, while Gohn et al. \cite{ghlssw} studied $b$-prodigious numbers, i.e., positive integers that are divisible by the product of their nonzero digits in base $b$.

One of the variations that we study are the Zeckendorf-Niven numbers, introduced by Ray and Cooper in 2006 \cite{rc}. To understand the definition, we begin with the following theorem by Zeckendorf \cite{zeckendorf}. Throughout this paper, we use $\N$ to denote the set of positive integers, and for all nonnegative integers $i$, $F_i$ denotes the $i$-th Fibonacci number, i.e., $F_0=0$, $F_1=1$, and $F_i=F_{i-1}+F_{i-2}$ for all integers $i\geq2$.

\begin{theorem}
For every $n\in\N$, there exists a unique sequence $(\zeta_i(n))_{i=2}^\infty$ such that $\zeta_i(n)\in\{0,1\}$ and $\zeta_i(n)+\zeta_{i+1}(n)\leq1$ for all $i\geq2$, and
\begin{equation}\label{eq:zeckendorf}
n=\sum_{i=2}^\infty\zeta_i(n)F_i.
\end{equation}
\end{theorem}

Note that the sum in equation~\eqref{eq:zeckendorf} starts at $i=2$, or else we clearly lose uniqueness (either from adding $F_0$ as a summand or from switching between $F_1$ and $F_2$.) This sum is called the \emph{Zeckendorf decomposition} of $n$. Although the Zeckendorf decomposition is presented as an infinite sum in equation~\eqref{eq:zeckendorf}, there are only finitely many nonzero summands since all Fibonacci numbers are positive. For every $n\in\N$, the number of nonzero summands in the Zeckendorf decomposition of $n$ is defined as $s_Z(n)$, i.e.,
\[s_Z(n)\coloneq\sum_{i=2}^\infty\zeta_i(n).\]
We say that $n$ is \emph{Zeckendorf-Niven} if $s_Z(n)$ divides $n$.

Similarly, we define Lucas-Niven numbers as follows. For all nonnegative integers $i$, let $L_i$ denote the $i$-th Lucas number, i.e., $L_0=2$, $L_1=1$, and $L_i=L_{i-1}+L_{i-2}$ for all integers $i\geq2$. Brown \cite{brown} proved that every positive integer has a unique \emph{Lucas decomposition}.

\begin{theorem}
For every $n\in\N$, there exists a unique sequence $(\lambda_i(n))_{i=0}^\infty$ such that $\lambda_i(n)\in\{0,1\}$ and $\lambda_i(n)+\lambda_{i+1}(n)\leq1$ for all $i\geq0$, $\lambda_0(n)+\lambda_2(n)\leq1$, and
\[n=\sum_{i=0}^\infty\lambda_i(n)L_i.\]
\end{theorem}

We say that $n$ is \emph{Lucas-Niven} if $s_L(n)$ divides $n$, where
\[s_L(n)\coloneq\sum_{i=0}^\infty\lambda_i(n).\]

A common direction of study related to $b$-Niven numbers is their distribution in arithmetic progressions. An \emph{arithmetic progression} is an infinite sequence of positive integers with the same common difference between consecutive terms. If the common difference is $d$, then we call this arithmetic progression a \emph{$d$-AP}. Cooper and Kennedy \cite{ck1993} proved that the maximum number of consecutive Niven terms in a $1$-AP is $20$, while Grundman \cite{g1994} and Wilson \cite{wilson} generalized this result and proved that the maximum number of consecutive $b$-Niven terms in a $1$-AP is $2b$. Later, Grundman et al.\ \cite{ghw} studied the maximum number of consecutive $b$-Niven terms in a $d$-AP when $d>1$. Harrington et al.\ \cite{hlw} also showed that every arithmetic progression contains infinitely many $b$-Niven numbers.

One of the key results on Zeckendorf-Niven numbers in arithmetic progressions is provided by Grundman \cite{g2007}, and she proved that any sequence of consecutive Zeckendorf-Niven numbers greater than $6$ has a maximum length of $4$. We show that there are infinitely many Zeckendorf-Niven numbers as well as infinitely many Lucas-Niven numbers in any given arithmetic progression, with the proofs presented in Sections~\ref{sec:infinZN} and \ref{sec:infinLN}. We also show in Section~\ref{sec:infinZN} that every arithmetic progression contains an arbitrary number of consecutive terms that are all non-Zeckendorf-Niven. In Section~\ref{sec:2AP}, we give upper and lower bounds on the maximum number of consecutive Zeckendorf-Niven terms in a $2$-AP. In Section~\ref{sec:Fd}, we prove that the maximum number of consecutive terms in an $F_d$-AP that share the same $s_Z$-value is three, and show that there are infinitely many of such occurrences where all three terms are also Zeckendorf-Niven.

\section{Infinitely many Zeckendorf-Niven numbers in every arithmetic progression}\label{sec:infinZN}

We begin this section with several preliminary results on Fibonacci numbers.

\begin{proposition}[\cite{koshy}, Theorem~16.1]\label{prop:Fi|Fj}
For all $i,j,\in\N$ such that $i\mid j$, we have $F_i\mid F_j$.
\end{proposition}

\begin{proposition}[\cite{wall}]\label{prop:Fiperiodic}
For all $d\in\N$, the sequence $(F_i)_{i=0}^\infty$ is periodic modulo $d$.
\end{proposition}

The period of the sequence $(F_i)_{i=0}^\infty$ modulo $d$ is called the \emph{$d$-th Pisano period} and is denoted by $\pi(d)$. Some of the Pisano periods are known.

\begin{proposition}[\cite{ehrlich}]\label{prop:Pisano}
For all $j\in\N$, we have $\pi(F_{2j})=4j$ and $\pi(F_{2j+1})=8j+4$.
\end{proposition}

The following lemma allows us to focus on arithmetic progressions that have a Fibonacci number as the common difference.

\begin{lemma}\label{lem:d=F2jsubseq}
For every $a,d\in\N$, there exists $j\in\N$ such that $F_{2j}\geq a$ and $(a+kF_{2j})_{k=0}^\infty$ is a subsequence of $(a+kd)_{k=0}^\infty$.
\end{lemma}
\begin{proof}
Let $j=a\pi(d)$. Then $F_{2j}\geq F_{2a}\geq a$. The conclusion follows since $F_{2j}=F_{2a\pi(d)}-F_0\equiv0\pmod{d}$ by the definition of Pisano period.
\end{proof}

The next lemma shows the existence of an appropriate term in the arithmetic progression $(a+kF_{2j})_{k=0}^\infty$ that achieves a prescribed number of nonzero summands in its Zeckendorf decomposition.

\begin{lemma}\label{lem:prescribedsZ}
Let $(a+kF_{2j})_{k=0}^\infty$ be a given arithmetic progression. Then for every integer $m>s_Z(a)$, there exists a term $N$ in the arithmetic progression $(a+kF_{2j})_{k=0}^\infty$ such that $s_Z(N)=m$.
\end{lemma}
\begin{proof}
Let $i_0\in\N$ such that $F_{2i_0j}\geq a$, and let
\[N=a+\sum_{i=i_0+1}^{i_0+m-s_Z(a)}F_{2ij}.\]
Then $N$ is a term in the arithmetic progression $(a+kF_{2j})_{k=0}^\infty$ by Proposition~\ref{prop:Fi|Fj}. Furthermore, there are no consecutive Fibonacci numbers inside the summation in the definition of $N$, so $s_Z(N)=s_Z(a)+m-s_Z(a)=m$ by the uniqueness of Zeckendorf decomposition.
\end{proof}

Now, we are ready to prove the main theorem of this section.

\begin{theorem}\label{thm:infinZN}
Every arithmetic progression contains infinitely many Zeckendorf-Niven numbers.
\end{theorem}
\begin{proof}
It suffices to show that for every $a,d\in\N$, there exists a Zeckendorf-Niven number in the arithmetic progression $(a+kd)_{k=0}^\infty$, since the original statement can then be proved inductively by replacing the value of $a$ with a larger value at each step. By Lemma~\ref{lem:d=F2jsubseq}, we can restrict our attention to arithmetic progressions of the form $(a+kF_{2j})_{k=0}^\infty$ where $F_{2j}\geq a$. Since $F_{2j+1}>F_{2j}\geq a\geq s_Z(a)$, there exists a term $N$ in the arithmetic progression $(a+kF_{2j})_{k=0}^\infty$ such that $s_Z(N)=F_{2j+1}$ by Lemma~\ref{lem:prescribedsZ}. To complete the proof, we are going to transform the Zeckendorf decomposition of $N$ term-by-term so that each new summand is divisible by $s_Z(N)$ and yet maintains the same residue modulo the common difference $F_{2j}$ of the arithmetic progression.

Let
\[N=\sum_{i=1}^{F_{2j+1}}F_{r_i}\]
be the Zeckendorf decomposition of $N$, where $r_1\geq2$ and $r_{i+1}-r_i\geq2$ for all $1\leq i\leq F_{2j+1}-1$. Since $2j+1$ and $4j$ are coprime, for each $r_i$, there exists an integer $0\leq m_i<2j+1$ such that $m_i\equiv-r_i(4j)^{-1}\pmod{2j+1}$, i.e., $2j+1\mid r_i+4jm_i$. Hence, $F_{2j+1}\mid F_{r_i+4jm_i}$ by Proposition~\ref{prop:Fi|Fj}. Furthermore, let $t_i=r_i+4jm_i+(2j+1)4ij$ for each $1\leq i\leq F_{2j+1}$, and define
\[N'=\sum_{i=1}^{F_{2j+1}}F_{t_i}.\]
Note that the summation provided is the Zeckendorf decomposition of $N'$ since $t_1\geq2$ and
\[t_{i+1}-t_i=r_{i+1}-r_i+4j(m_{i+1}-m_i+2j+1)>r_{i+1}-r_i\geq2\]
for all $1\leq i\leq F_{2j+1}-1$, thus $s_Z(N')=F_{2j+1}$. By Proposition~\ref{prop:Pisano}, we have $\pi(F_{2j+1})=8j+4$, which divides $t_i-(r_i+4jm_i)$ for all $1\leq i\leq F_{2j+1}$, so $F_{t_i}\equiv F_{r_i+4jm_i}\equiv0\pmod{F_{2j+1}}$. Therefore, $F_{2j+1}\mid N'$ and $N'$ is Zeckendorf-Niven.

Finally, by Proposition~\ref{prop:Pisano}, $\pi(F_{2j})=4j$, which divides $t_i-r_i$ for all $1\leq i\leq F_{2j+1}$, implying that $F_{2j}\mid F_{t_i}-F_{r_i}$. Hence, $F_{2j}\mid N'-N$, i.e., $N'$ is a term in the arithmetic progression $(a+kF_{2j})_{k=0}^\infty$. 
\end{proof}

Before ending this section, we provide a result concerning the distribution of non-Zeckendorf-Niven numbers in arithmetic progressions.

\begin{theorem}\label{thm:nonZN}
Every arithmetic progression contains infinitely many subsequences of $\ell$ consecutive non-Zeckendorf-Niven terms, where $\ell$ is an arbitrary positive integer.
\end{theorem}
\begin{proof}
Let $(a+kd)_{k=0}^\infty$ be a given arithmetic progression. Let $g\in\N$ such that $g>a+(\ell-1)d$, and let
\[j=\lcm\big(\{\pi(s_Z(a+kd)+g):0\leq k\leq\ell-1\}\cup\{\pi(d)\}\big).\]
Note that $j\geq2$ since $\pi(n)\geq2$ for all integers $n\geq2$. Further let $m\in\N$ such that $jm\geq r+2$, where $F_r$ is the largest summand in the Zeckendorf decomposition of $a+(\ell-1)d$.

For each integer $0\leq k\leq\ell-1$, let
\[a_k=a+kd+\sum_{i=1}^gF_{j(m+i)}.\]
If we replace $a+kd$ by its Zeckendorf decomposition, then the above summation is the Zeckendorf decomposition of $a_k$ due to the construction of $jm$ and the fact that $j\geq2$. Hence, $s_Z(a_k)=s_Z(a+kd)+g$, implying that $s_Z(a_k)\mid\sum_{i=1}^gF_{j(m+i)}$ by the definition of $j$. However, $s_Z(a_k)>g\geq a+kd>0$, so $s_Z(a_k)\nmid a_k$, i.e., $a_k$ is non-Zeckendorf-Niven. The conclusion follows since $a_k-a=kd+\sum_{i=1}^gF_{j(m+i)}\equiv0\pmod{d}$.
\end{proof}

\section{Infinitely many Lucas-Niven numbers in every arithmetic progression}\label{sec:infinLN}

To prove a result analogous to Theorem~\ref{thm:infinZN} for the Lucas sequence, we need to make small modifications to our technique due to the lack of divisibility properties in the Lucas sequence such as those given by Proposition~\ref{prop:Fi|Fj} and Lemma~\ref{lem:d=F2jsubseq}.

Similar to the Fibonacci sequence, it is known that the Lucas sequence is also periodic modulo $d$ for all $d\in\N$ \cite{wall}. Denote this period by $\pi_L(d)$. In the following proposition, we are going to state an interesting relation between $\pi(d)$ and $\pi_L(d)$ as well as some property of $\pi(d)$, which directly implies a subsequent lemma on $\pi_L(d)$.

\begin{proposition}[\cite{wall}]\label{prop:piandpiL}\ 
\begin{enumerate}[$(i)$]
\item For every $d\in\N$ such that $d\not\equiv0\pmod{5}$, we have $\pi(d)=\pi_L(d)$.
\item\label{item:piL} For every prime $p$ such that $p\equiv-1\pmod{10}$, we have $\pi(p)\mid p-1$.
\end{enumerate}
\end{proposition}

\begin{lemma}\label{lem:piL}
For every prime $p$ such that $p\equiv-1\pmod{10}$, we have $\pi_L(p)\mid p-1$.
\end{lemma}

The next lemma is an analogue of Lemma~\ref{lem:prescribedsZ}.

\begin{lemma}\label{lem:prescribedsL}
Let $(a+kd)_{k=0}^\infty$ be a given arithmetic progression. Then for every integer $m>s_L(a)$, there exists a term $N$ in the arithmetic progression $(a+kd)_{k=0}^\infty$ such that $s_L(N)=m$.
\end{lemma}
\begin{proof}
It suffices to show that there exists a term $N$ in the arithmetic progression $(a+kd)_{k=0}^\infty$ such that $s_L(N)=s_L(a)+1$, since the original statement can then be proved inductively by replacing the value of $a$ with $N$ at each step.

Let
\[a=\sum_{i=1}^{s_L(a)}L_{r_i}\]
be the Lucas decomposition of $a$. To simply the notation, let $r=r_{s_L(a)}$. Define
\[N=L_{r-1+3\pi_L(d)}+L_{r-2+2\pi_L(d)}+\sum_{i=1}^{s_L(a)-1}L_{r_i}.\]
Note that the summation provided is the Lucas decomposition of $N$ since $(r-1+3\pi_L(d))-(r-2+2\pi_L(d))\geq2$ and $(r-2+2\pi_L(d))-r_{s_L(a)-1}\geq(r-2+2\pi_L(d))-(r-2)\geq2$, thus $s_L(N)=s_L(a)+1$. Furthermore,
\[N-a=L_{r-1+3\pi_L(d)}+L_{r-2+2\pi_L(d)}-L_r=(L_{r-1+3\pi_L(d)}-L_{r-1})+(L_{r-2+2\pi_L(d)}-L_{r-2}),\]
which is divisible by $d$. Therefore, $N$ is a term in the arithmetic progression $(a+kd)_{k=0}^\infty$.
\end{proof}

This leads to the main theorem of this section.

\begin{theorem}\label{thm:infinLN}
Every arithmetic progression contains infinitely many Lucas-Niven numbers.
\end{theorem}
\begin{proof}
With the same logic presented in the proof of Theorem~\ref{thm:infinZN}, it suffices to show that for every $a,d\in\N$, there exists a Lucas-Niven number in the arithmetic progression $(a+kd)_{k=0}^\infty$. By Dirichlet's Theorem, there are infinitely many primes $p$ such that $p\equiv-1\pmod{10\pi_L(d)}$. Let $p>s_L(a)$ be such a prime. Note that $\gcd(p-1,\pi_L(d))\leq2$, which implies that $\gcd(\pi_L(p),\pi_L(d))\leq2$ by Lemma~\ref{lem:piL}. Since $p>s_L(a)$, there exists a term $N$ in the arithmetic progression $(a+kd)_{k=0}^\infty$ such that $s_L(N)=p$ by Lemma~\ref{lem:prescribedsL}. Let
\[N=\sum_{i=1}^pL_{r_i}\]
be the Lucas decomposition of $N$. We are going to complete this proof by considering the following two cases.\\

\noindent\textit{Case $1$}: $\gcd(\pi_L(p),\pi_L(d))=1$.\\

For each $1\leq i\leq p$, there exists an integer $0\leq m_i<\pi_L(p)$ such that
\[m_i\equiv(-r_i+1)\pi_L(d)^{-1}\pmod{\pi_L(p)},\]
i.e., $\pi_L(p)\mid r_i-1+\pi_L(d)m_i$. Let $t_i=r_i+\pi_L(d)(m_i+i\pi_L(p))$ and define
\[N'=\sum_{i=1}^pL_{t_i}.\]
Note that $\pi_L(p)\geq2$, so the summation provided is the Lucas decomposition of $N'$ since $t_1\geq2$ and $t_{i+1}-t_i=r_{i+1}-r_i+\pi_L(d)\pi_L(p)\geq2$ for all $1\leq i\leq p-1$. Consequently, $s_L(N')=p$. Moreover, $t_i\equiv 1\pmod{\pi_L(p)}$ for all $1\leq i\leq p$, so $N'\equiv\sum_{i=1}^pL_1\equiv0\pmod{p}$, implying that $N'$ is Lucas-Niven. Furthermore, since $\pi_L(d)\mid t_i-r_i$ for all $1\leq i\leq p$, we have $d\mid N'-N$. Therefore, $N'$ is a term in the arithmetic progression $(a+kd)_{k=0}^\infty$.\\

\noindent\textit{Case $2$}: $\gcd(\pi_L(p),\pi_L(d))=2$.\\

Let $E=\{1\leq i\leq p:r_i\text{ is even}\}$ and $O=\{1\leq i\leq p:r_i\text{ is odd}\}$. We consider the following subcases.\\

\noindent\textit{Case $2.1$}: $|E|\geq|O|$.\\

Partition $E$ into $E_1$ and $E_2$ such that $|E_1|=|E|-|O|$ and $|E_2|=|O|$. For each $1\leq i\leq p$, there exists an integer $0\leq m_i<\pi_L(p)/2$ such that
\begin{itemize}
\item $m_i\equiv-(r_i/2)(\pi_L(d)/2)^{-1}\pmod{\pi_L(p)/2}$ if $i\in E_1$, i.e., $\pi_L(p)\mid r_i+\pi_L(d)m_i$;
\item $m_i\equiv-((r_i-2)/2)(\pi_L(d)/2)^{-1}\pmod{\pi_L(p)/2}$ if $i\in E_2$, i.e., $\pi_L(p)\mid r_i-2+\pi_L(d)m_i$; and
\item $m_i\equiv-((r_i-1)/2)(\pi_L(d)/2)^{-1}\pmod{\pi_L(p)/2}$ if $i\in O$, i.e., $\pi_L(p)\mid r_i-1+\pi_L(d)m_i$.
\end{itemize}
Similar to Case $1$, let $t_i=r_i+\pi_L(d)(m_i+i\pi_L(p))$ and define $N'=\sum_{i=1}^pL_{t_i}$. With the same reasoning as in Case $1$, we have $s_L(N')=p$, and we also know that $N'$ is a term in the arithmetic progression $(a+kd)_{k=0}^\infty$. Moreover, $t_i\equiv0\pmod{\pi_L(p)}$ if $i\in E_1$, $t_i\equiv2\pmod{\pi_L(p)}$ if $i\in E_2$, and $t_i\equiv1\pmod{\pi_L(p)}$ if $i\in O$, so
\[N'\equiv\sum_{i\in E_1}L_0+\sum_{i\in E_2}L_2+\sum_{i\in O}L_1\equiv2(|E|-|O|)+3|O|+|O|\equiv2p\equiv0\pmod{p},\]
implying that $N'$ is Lucas-Niven.\\

\noindent\textit{Case $2.2$}: $|E|<|O|$.\\

Partition $O$ into $O_1$ and $O_2$ such that $|O_1|=|O|-|E|$ and $|O_2|=|E|$. For each $1\leq i\leq p$, there exists an integer $0\leq m_i<\pi_L(p)/2$ such that
\begin{itemize}
\item $m_i\equiv-((r_i-2)/2)(\pi_L(d)/2)^{-1}\pmod{\pi_L(p)/2}$ if $i\in E$, i.e., $\pi_L(p)\mid r_i-2+\pi_L(d)m_i$;
\item $m_i\equiv-((r_i-1)/2)(\pi_L(d)/2)^{-1}\pmod{\pi_L(p)/2}$ if $i\in O_1$, i.e., $\pi_L(p)\mid r_i-1+\pi_L(d)m_i$; and
\item $m_i\equiv-((r_i+1)/2)(\pi_L(d)/2)^{-1}\pmod{\pi_L(p)/2}$ if $i\in O_2$, i.e., $\pi_L(p)\mid r_i+1+\pi_L(d)m_i$.
\end{itemize}
Once again, letting $t_i=r_i+\pi_L(d)(m_i+i\pi_L(p))$ and defining $N'=\sum_{i=1}^pL_{t_i}$ give us $s_L(N')=p$ and that $N'$ is a term in the arithmetic progression $(a+kd)_{k=0}^\infty$. Finally, $t_i\equiv2\pmod{\pi_L(p)}$ if $i\in E$, $t_i\equiv1\pmod{\pi_L(p)}$ if $i\in O_1$, and $t_i\equiv-1\pmod{\pi_L(p)}$ if $i\in O_2$, so by defining $L_{-1}=-1$, we have
\[N'\equiv\sum_{i\in E}L_2+\sum_{i\in O_1}L_1+\sum_{i\in O_2}L_{-1}\equiv3|E|+(|O|-|E|)+(-1)|E|\equiv p\equiv0\pmod{p},\]
implying that $N'$ is Lucas-Niven.
\end{proof}

\section{Maximum number of consecutive Zeckendorf-Niven terms in a $2$-AP}\label{sec:2AP}

Motivated by Grundman's proof \cite{g2007} on the maximum number of consecutive Zeckendorf-Niven terms in a $1$-AP, we define $z_6(n)=\sum_{i=2}^6\zeta_i(n)F_i$ for every $n\in\N$. In the following two lemmas, we demonstrate some interesting relation between the values of $z_6(n)$ and Zeckendorf-Niven numbers.

\begin{lemma}\label{lem:z6=4or9}
For every integer $n\geq10$, if $z_6(n)=7$ and $\zeta_7(n)=0$, $z_6(n)=4$, or $z_6(n)=9$, then $n$ and $n+2$ cannot be both Zeckendorf-Niven.
\end{lemma}
\begin{proof}
If $z_6(n)=7=F_3+F_5$ and $\zeta_7(n)=0$, then $z_6(n+2)=9=F_2+F_6$; if $z_6(n)=4=F_2+F_4$, then $\zeta_6(n)=0$, so $z_6(n+2)=6=F_2+F_5$; if $z_6(n)=9=F_2+F_6$, then $\zeta_7(n)=0$, so $z_6(n+2)=11=F_4+F_6$. Hence, $s_Z(n)=s_Z(n+2)$ in all cases. Since $n\geq10$ and $z_6(n)\in\{7,4,9\}$, we have $s_Z(n)\geq3$. As a result, $s_Z(n)$ does not divide $(n+2)-n=2$, so $n$ and $n+2$ cannot be both Zeckendorf-Niven.
\end{proof}

\begin{lemma}\label{lem:z6=1or3}
For every integer $n\geq4$, if $z_6(n)=1$ or $z_6(n)=3$ and both $n$ and $n+2$ are Zeckendorf-Niven, then $s_Z(n)=2$.
\end{lemma}
\begin{proof}
If $z_6(n)=1=F_2$, then $z_6(n+2)=3=F_4$ since $\zeta_5(n)=\zeta_5(n+2)=0$; if $z_6(n)=3=F_4$, then $z_6(n)=5=F_5$ since $\zeta_6(n)=\zeta_6(n+2)=0$. Hence, $s_Z(n)=s_Z(n+2)$. If both $n$ and $n+2$ are Zeckendorf-Niven, then $s_Z(n)$ divides $(n+2)-n=2$. Since $n\geq4$ and $z_6(n)=\{1,3\}$, we have $s_Z(n)\geq2$. Therefore, $s_Z(n)=2$.
\end{proof}

The following two theorems establish an upper bound and a lower bound for the maximum number of consecutive Zeckendorf-Niven terms in a $2$-AP.

\begin{theorem}\label{thm:upperbd2AP}
The only sequences of eight or more consecutive Zeckendorf-Niven terms in a $2$-AP are subsequences of $2,4,6,8,10,12,14,16,18$.
\end{theorem}
\begin{proof}
It is easy to verify that $2,4,6,8,10,12,14,16,18$ are Zeckendorf-Niven. Now, suppose that $(n+2k)_{k=0}^7$ is not a subsequence of $2,4,6,8,10,12,14,16,18$ and every term is Zeckendorf-Niven. Since $7,19,20$ are not Zeckendorf-Niven, we have $n\geq21$.

By Lemma~\ref{lem:z6=4or9}, $z_6(n+2k)\notin\{4,9\}$ for any integer $0\leq k\leq6$. We also have $z_6(n+2k)\neq2$ for any integer $0\leq k\leq5$; otherwise, $z_6(n+2k+2)=4$ since $\zeta_5(n+2k)=\zeta_5(n+2k+2)=0$. Similarly, we have $z_6(n+2k)\neq0$ for any integer $0\leq k\leq4$; otherwise, $z_6(n+2k+2)=2$ since $\zeta_4(n+2k)=\zeta_4(n+2k+2)=0$. If $z_6(n)=11=F_4+F_6$, then the Zeckendorf decomposition of $n+2$ is
\[F_j+\sum_{i=j+2}^\infty\zeta_i(n)F_i,\]
where $j\geq7$ is the smallest integer such that $\zeta_j(n)=\zeta_{j+1}(n)=0$. This leads to a contradiction since $z_6(n+2)=0$, so we have $z_6(n)\neq11$.

By Lemma~\ref{lem:z6=4or9} again, $z_6(n+2k)\neq7$ for any integer $0\leq k\leq 6$ if $\zeta_7(n+2k)=0$. If $z_6(n+2k)=3$ for some integer $0\leq k\leq4$, then $s_Z(n+2k)=2$ by Lemma~\ref{lem:z6=1or3}. As a result, we have $\zeta_7(n+2k)=\zeta_7(n+2k+4)=0$ since $n\geq21$ and $z_6(n+2k+4)=7$, which is a contradiction. Hence, $z_6(n+2k)\neq3$ for any integer $0\leq k\leq4$. From this, we deduce that $z_6(n+2k)\neq1$ for any integer $0\leq k\leq3$; otherwise, $z_6(n+2k+2)=3$ since $\zeta_5(n+2k)=\zeta_5(n+2k+2)=0$.

Next, if $z_6(n+2k)=12=F_2+F_4+F_6$ for some integer $0\leq k\leq2$, then the Zeckendorf decomposition of $n+2k+2$ is
\[F_2+F_j+\sum_{i=j+2}^\infty\zeta_i(n+2k)F_i,\]
where $j\geq7$ is the smallest integer such that $\zeta_j(n+2k)=\zeta_{j+1}(n+2k)=0$. This leads to a contradiction since $z_6(n+2k+2)=1$, so we have $z_6(n+2k)\neq12$ for any integer $0\leq k\leq2$. Consequently, $z_6(n+2k)\neq10$ for any integer $0\leq k\leq1$; otherwise, $z_6(n+2k+2)=12$ since $\zeta_7(n+2k)=\zeta_7(n+2k+2)=0$. Furthermore, $z_6(n)\neq8$; otherwise, $z_6(n+2)=10$ since $\zeta_7(n)=\zeta_7(n+2)=0$.

If $z_6(n+2k)=7$ for some integer $0\leq k\leq2$ and $\zeta_7(n+2k)=1$, then the Zeckendorf decomposition of $n+2k+2$ is
\[F_2+F_j+\sum_{i=j+2}^\infty\zeta_i(n+2k)F_i,\]
where $j\geq8$ is the smallest integer such that $\zeta_j(n+2k)=\zeta_{j+1}(n+2k)=0$. This again leads to a contradiction since $z_6(n+2k+2)=1$, so we have $z_6(n+2k)\neq7$ for any integer $0\leq k\leq2$. Then we also have $z_6(n)\neq5$; otherwise, $z_6(n+2)=7$ since $\zeta_6(n)=\zeta_6(n+2)=0$.

The only remaining case is $z_6(n)=6=F_2+F_5$. If $\zeta_7(n)=1$, then the Zeckendorf decomposition of $n+2$ is
\[F_j+\sum_{i=j+2}^\infty\zeta_i(n)F_i,\]
where $j\geq8$ is the smallest integer such that $\zeta_j(n)=\zeta_{j+1}(n)=0$. This leads to a contradiction since $z_6(n+2)=0$. If $\zeta_7(n)=0$, then $z_6(n+4)=10=F_3+F_6$, so $s_Z(n)=s_Z(n+4)$. Since $n\geq21$ and $z_6(n)=6$, we have $s_Z(n)\geq3$, thus $s_Z(n)=4$. On the other hand, the Zeckendorf decomposition of $n+8$ is
\[F_2+F_j+\sum_{i=j+2}^\infty\zeta_i(n)F_i,\]
where $j\geq7$ is the smallest integer such that $\zeta_j(n)=\zeta_{j+1}(n)=0$. Hence, $z_6(n+8)=1$, which implies that $s_Z(n+8)=2$ by Lemma~\ref{lem:z6=1or3}. The only possibility is $n=F_2+F_5+F_8+F_{10}=82$. However, $s_Z(n)=4\nmid n$, contradicting that $n$ is Zeckendorf-Niven.
\end{proof}

\begin{theorem}
For every $j\in\N$, let $n=27+F_{120j+17}$. Then $n,n+2,n+4,n+6,n+8$ is a sequence of five consecutive Zeckendorf-Niven terms in a $2$-AP.
\end{theorem}
\begin{proof}
The Zeckendorf decomposition of $n$ is $F_2+F_5+F_8+F_{120j+17}$, and it is easy to verify that $s_Z(n)=s_Z(n+4)=4$, $s_Z(n+2)=s_Z(n+8)=3$, and $s_Z(n+6)=5$. Note that $\pi(3)=8$, $\pi(4)=6$, and $\pi(5)=20$, so $\lcm\{\pi(3),\pi(4),\pi(5)\}=120$, implying that $F_{120j+17}\equiv F_{17}=1597\equiv37\pmod{3\cdot4\cdot5}$. The result follows since $4\mid 27+37$, $3\mid 29+37$, $4\mid31+37$, $5\mid33+37$, and $3\mid35+37$.
\end{proof}

\section{Consecutive Zeckendorf-Niven terms in an $F_d$-AP with the same $s_Z$-value}\label{sec:Fd}

Other than $1$-APs and $2$-APs, the study of consecutive Zeckendorf-Niven numbers in a general $d$-AP is difficult. Here, we choose a special case to consider, namely when the common difference is a Fibonacci number, denoted by $F_d$. We further restrict our attention to the scenario when the consecutive terms share the same $s_Z$-value.

Generalizing the notation from Section~\ref{sec:2AP}, we define $z_{x,y}(n)=\sum_{i=x}^y\zeta_i(n)F_i$ for all $n,x,y\in\N$ (if $i<2$, then we define $\zeta_i(n)=0$ for all $n\in\N$.) Now, we are ready to present the two main results in this section.

\begin{theorem}\label{thm:samesZ}
For every integer $d\geq2$, there exists infinitely many $a\in\N$ such that $s_Z(a)=s_Z(a+F_d)=s_Z(a+2F_d)$. Furthermore, there does not exist $a\in\N$ such that $s_Z(a)=s_Z(a+F_d)=s_Z(a+2F_d)=s_Z(a+3F_d)$.
\end{theorem}
\begin{proof}
For every $a\in\N$ such that $z_{d-2,d+3}(a)=F_{d-1}$, since $z_{d-2,d+3}(a+F_d)=F_{d+1}$ and $z_{d-2,d+3}(a+2F_d)=F_{d+2}$, we have $s_Z(a)=s_Z(a+F_d)=s_Z(a+2F_d)$.

Now, suppose that $s_Z(a)=s_Z(a+F_d)=s_Z(a+2F_d)=s_Z(a+3F_d)$ for some $a\in\N$. Note that $z_{d-1,d+1}(a+kF_d)\neq0$ for all $k\in\{0,1,2\}$; otherwise, $s_Z(a+(k+1)F_d)=s_Z(a+kF_d)+1$. Consequently, we also have $z_{d-1,d+1}(a+kF_d)\neq F_{d+1}$ for all $k\in\{0,1\}$; otherwise, $z_{d-1,d+1}(a+(k+1)F_d)=0$. Furthermore, $z_{d-1,d+1}(a)\neq F_{d-1}$; otherwise, $z_{d-1,d+1}(a+F_d)\in\{0,F_{d+1}\}$. If $z_{d-1,d+1}(a+kF_d)=F_{d-1}+F_{d+1}$ for some $k\in\{0,1\}$, then the Zeckendorf decomposition of $a+(k+1)F_d$ is
\[\sum_{i=2}^{d-3}\zeta_i(a+kF_d)F_i+F_{d-1}+F_j+\sum_{i=j+2}^\infty\zeta_i(a+kF_d)F_i,\]
where $j\geq d+2$ is the smallest integer such that $\zeta_j(a+kF_d)=\zeta_{j+1}(a+kF_d)=0$. Since $s_Z(a+kF_d)=s_Z(a+(k+1)F_d)$, we have $j=d+2$. As a result, the Zeckendorf decomposition of $a+(k+2)F_d$ is
\[\sum_{i=2}^{d-3}\zeta_i(a+kF_d)F_i+F_{j'}+\sum_{i=j'+2}^\infty\zeta_i(a+kF_d)F_i,\]
where $j'\geq d+3$ is the smallest integer such that $\zeta_{j'}(a+kF_d)=\zeta_{j'+1}(a+kF_d)=0$. This leads to a contradiction since $s_Z(a+kF_d)>s_Z(a+(k+2)F_d)$. Thus $z_{d-1,d+1}(a+kF_d)\neq F_{d-1}+F_{d+1}$ for all $k\in\{0,1\}$. Therefore, $z_{d-1,d+1}(a)=F_d$.

It is useful to observe that $2F_i=F_{i-2}+F_{i+1}$ for all integers $i\geq2$. Next, we have $z_{d-3,d+1}(a)\neq F_d$; otherwise, $z_{d-1,d+1}(a+F_d)\in\{0,F_{d+1}\}$. If $z_{d-3,d+1}(a)=F_{d-3}+F_d$, then the Zeckendorf decomposition of $a+F_d$ is
\[\sum_{i=2}^{d-5}\zeta_i(a)F_i+F_{d-1}+F_j+\sum_{i=j+2}^\infty\zeta_i(a)F_i,\]
where $j\geq d+1$ is the smallest integer such that $\zeta_j(a)=\zeta_{j+1}(a)=0$. Since $s_Z(a)=s_Z(a+F_d)$, we have $j=d+1$. This leads to a contradiction since $z_{d-1,d+1}(a+F_d)=F_{d-1}+F_{d+1}$. Hence, it remains to consider the case when $z_{d-3,d+1}(a)=F_{d-2}+F_d$.

Since $F_{d-2}+2F_d=2F_{d-2}+F_{d+1}=F_{d-4}+F_{d-1}+F_{d+1}$, we have $z_{d-1,d+1}(a+F_d)\in\{F_{d-1},F_{d-1}+F_{d+1}\}$. We have shown that $z_{d-1,d+1}(a+F_d)\neq F_{d-1}+F_{d+1}$, so $z_{d-1,d+1}(a+F_d)=F_{d-1}$, which implies that $\zeta_{d+2}(a)=1$. Let $j\geq d+3$ be the smallest integer such that $\zeta_j(a)=\zeta_{j+1}(a)=0$, and let $j'\leq d-3$ be the largest integer such that $\zeta_{j'}(a)=\zeta_{j'-1}(a)=0$. If $\zeta_{j'-2}(a)=1$, then the Zeckendorf decomposition of $a$ is
\[\sum_{i=2}^{j'-4}\zeta_i(a)F_i+F_{j'-2}+\sum_{i=0}^{\frac{d-2-(j'+1)}{2}}F_{j'+1+2i}+F_d+\sum_{i=0}^{\frac{j-1-(d+2)}{2}}F_{d+2+2i}+\sum_{i=j+2}^\infty\zeta_i(a)F_i\]
and that of $a+F_d$ is
\[\sum_{i=2}^{j'-4}\zeta_i(a)F_i+F_{j'}+\sum_{i=0}^{\frac{d-2-(j'+1)}{2}}F_{j'+2+2i}+F_j+\sum_{i=j+2}^\infty\zeta_i(a)F_i.\]
This leads to a contradiction since $s_Z(a)>s_Z(a+F_d)$. Thus $\zeta_{j'-2}(a)=0$, the Zeckendorf decomposition of $a$ is
\[\sum_{i=2}^{j'-3}\zeta_i(a)F_i+\sum_{i=0}^{\frac{d-2-(j'+1)}{2}}F_{j'+1+2i}+F_d+\sum_{i=0}^{\frac{j-1-(d+2)}{2}}F_{d+2+2i}+\sum_{i=j+2}^\infty\zeta_i(a)F_i,\]
and that of $a+F_d$ is
\[\sum_{i=2}^{j'-3}\zeta_i(a)F_i+F_{j'-1}+\sum_{i=0}^{\frac{d-2-(j'+1)}{2}}F_{j'+2+2i}+F_j+\sum_{i=j+2}^\infty\zeta_i(a)F_i.\]
Since $s_Z(a)=s_Z(a+F_d)$, we have $j=d+3$. Then the Zeckendorf decomposition of $a+2F_d$ is
\[\sum_{i=2}^{j'-3}\zeta_i(a)F_i+F_{j'-1}+\sum_{i=0}^{\frac{d-4-(j'+1)}{2}}F_{j'+2+2i}+F_{d+1}+F_{d+3}+\sum_{i=d+5}^\infty\zeta_i(a)F_i\]
and that of $a+3F_d$ is
\[\sum_{i=2}^{j'-3}\zeta_i(a)F_i+F_{j'-1}+\sum_{i=0}^{\frac{d-4-(j'+1)}{2}}F_{j'+2+2i}+F_{j''}+\sum_{i=j''+2}^\infty\zeta_i(a)F_i,\]
where $j''\geq d+4$ is the smallest integer such that $\zeta_{j'}(a)=\zeta_{j'+1}(a)=0$. This also leads to a contradiction since $s_Z(a)>s_Z(a+3F_d)$.
\end{proof}

\begin{theorem}
For every integer $d\geq3$, there exists infinitely many $a\in\N$ such that $a$, $a+F_d$, and $a+2F_d$ are Zeckendorf-Niven and $s_Z(a)=s_Z(a+F_d)=s_Z(a+2F_d)$.
\end{theorem}
\begin{proof}
Let $m\geq2$ be a factor of $F_d$. Then for any nonnegative integer $j$, let
\[a=F_{d-1}+\sum_{i=1}^{m-1}F_{d+2+\pi(m)(i+j)}.\]
Since $\pi(m)\geq2$, this is the Zeckendorf decomposition of $a$. Moreover, the Zeckendorf decompositions of $a+F_d$ and $a+2F_d$ are
\[a+F_d=F_{d+1}+\sum_{i=1}^{m-1}F_{d+2+\pi(m)(i+j)}\]
and
\[a+2F_d=F_{d+2}+\sum_{i=1}^{m-1}F_{d+2+\pi(m)(i+j)},\]
respectively. Hence, $s_Z(a)=s_Z(a+F_d)=s_Z(a+2F_d)=m$. Since $m$ divides $F_d$, in order for $a$, $a+F_d$, and $a+2F_d$ to be Zeckendorf-Niven, it suffices to show that $m$ divides $a$. This follows since
\[a\equiv F_{d-1}+(m-1)F_{d+2}=F_{d-1}+(m-1)(F_{d-1}+2F_d)=mF_{d-1}+2(m-1)F_d\equiv0\pmod{m},\]
where the first congruence is due to the definition of the Pisano period.
\end{proof}

\section{Acknowledgement}\label{sec:acknowledgement}

These results are based on work supported by the National Science Foundation under grant numbered DMS-2341670. Also, special thanks to the organizers of the Polymath Jr for making this project possible.

\end{document}